\documentclass[12pt]{amsart}
\usepackage{amssymb,upref,enumerate}
\usepackage{tikz}
\usepackage{amsmath,amssymb,amscd,amsthm,amsxtra,bbm,mathrsfs}

\def\Xint#1{\mathchoice
{\XXint\displaystyle\textstyle{#1}}%
{\XXint\textstyle\scriptstyle{#1}}%
{\XXint\scriptstyle\scriptscriptstyle{#1}}%
{\XXint\scriptscriptstyle\scriptscriptstyle{#1}}%
\!\int}
\def\XXint#1#2#3{{\setbox0=\hbox{$#1{#2#3}{\int}$}
\vcenter{\hbox{$#2#3$}}\kern-.5\wd0}}

\def\dashint{\Xint-}

\allowdisplaybreaks

\def\avgB{\frac{1}{|B|}\int_B}
\def\avgBjko{\frac{1}{|{B_j^k}|}\int_{{B_j^k}}}
\def\avgBjk{\frac{1}{|\tilde{B_j^k}|}\int_{\tilde{B_j^k}}}
\def\R{\mathbb{R}}

\def\Z{\mathbb{Z}}
\def\D{\mathsf{D}}
\def\B{\mathcal{B}}

\def\N{\mathbb{N}}

\def\ra{\rightarrow}
\def\al{\alpha}

\def\ep{\epsilon}
\def\l{\lambda}

\def\m{\mu}

\def\s{\sigma}
\def\x{\chi}
\def\X{\Chi}

\def\W{\Omega}

\def\Sc{\mathcal S}

\def\1{\mathbbm{1}}
\def\X{\mathcal{X}}

\theoremstyle{plain}

\title{On the harmonic and geometric maximal operators}
\author{Linden Anne Duffee}
\address{Department of Mathematics\\
University of Alabama, Box 870350, 345 Gordon Palmer Hall.}\email{linden.duffee@gmail.com}

\author{Kabe Moen}
\address{Department of Mathematics\\
University of Alabama, Box 870350, 345 Gordon Palmer Hall.}\email{kabe.moen@ua.edu}


\newtheorem{mydef}{Definition}
\newtheorem{theorem}{Theorem}[section]
\newtheorem{lemma}[theorem]{Lemma}

\begin{document}
\maketitle

\begin{abstract}
We examine the harmonic and geometric maximal operators defined for a general basis of open sets in $\R^n$.  We prove two weight norm inequalities for the harmonic maximal operator assuming testing conditions over characteristic functions of unions of sets from the basis.  We also prove a that a bumped two weight $A_p$-like condition is sufficient for the two weight boundedness of the harmonic maximal operator.  

\end{abstract}

\section{Introduction}
We will study the harmonic maximal operator
$$M_{-1} f(x)=\sup_{Q\ni x} \left(\,\dashint_Q |f|^{-1}\,\right)^{-1}$$
and the geometric maximal operator
$$M_0f(x)=\sup_{Q\ni x}\exp\left(\,\dashint_Q\log|f|\,\right)$$
where the notation $\dashint_Qf$ denotes the average $\frac{1}{|Q|}\int_Qf$ and the supremum is over all cubes in $\R^n$ with sides parallel to the coordinate axes that contain the given point $x$.  Below we will be interested more general bases.  In the definition of the harmonic maximal operator we will use the conventions that $1/0=\infty$ and $1/\infty=0$.  The harmonic and geometric maximal operators are related to the Hardy-Littlewood maximal operator in the same way that the harmonic, geometric, and arithmetic means are related: If $x_1,\ldots,x_n$ are positive numbers then
$$\left(\frac1n\sum_{k=1}^n x_k^{-1}\right)^{-1}\leq \exp\left(\frac1n\sum_{k=1}^n \log x_k\right)\leq \frac{1}{n}\sum_{k=1}^n x_k.$$
The same inequalities hold for integral averages, and in particular, we have the pointwise bound
$$M_{-1}f(x)\leq M_0 f(x)\leq Mf(x),$$
where $M$ is the Hardy-Littlewood maximal operator
$$Mf(x)=\sup_{Q\ni x}\,\dashint_Q |f|.$$

A weight is a non-negative locally integrable function.  Given $p$, $1<p<\infty$, we say that a weight belongs to the class $A_p$ if 
\begin{equation} \label{Ap} [w]_{A_p}=\sup_Q \left(\,\dashint_Q w\,\right)\left(\dashint_Q w^{-\frac{1}{p-1}}\,\right)^{p-1}<\infty.\end{equation}
The class $A_1$ is the set of weights such that 
$$Mw(x)\leq Cw(x) \qquad \text{a.e.}\ x.$$
Finally the class $A_\infty$ will be the union of all $A_p$ class but can also be defined by the constant 
\begin{equation} \label{Ainfty} [w]_{A_\infty}=\sup_Q\left(\,\dashint_Q w\,\right)\exp\left(-\dashint_Q \log w\,\right)<\infty.\end{equation}

We will be interested in weighted estimates for the operators $M_{-1}$ and $M_0$. The study of such estimates was initiated by Shi \cite{Shi} who proved that $M_0$ was bounded on $L^p(w)$ for any $p>0$ when $w\in A_\infty$.  Cruz-Uribe and Neugebauer \cite{CN1} were the first to study the harmonic maximal operator.  Actually, they were interested in the minimal operator:
$$\mathsf{m}f(x)=\inf_{Q\ni x}\,\dashint_Q |f|\,.$$
However, Cruz-Uribe \cite{CU} points out that the minimal operator is simply the harmonic maximal operator in disguise:
$$M_{-1}f=\mathsf{m}(|f|^{-1})^{-1}.$$
We now state the results in \cite{CN1} recast in terms of the harmonic maximal operator.  
\begin{theorem}[\cite{CN1}] Given a weight $w$ and $0<p<\infty$, the following are equivalent:
\begin{enumerate}
\item $w\in A_\infty$;
\item the operator is of weak-type $(p,p)$
$$w(\{x\in \R^n: M_{-1}f(x)>\lambda\})\leq \frac{C}{\lambda^p}\int_{\R^n}|f|^pw\,;$$
\item the operator is of strong-type $(p,p)$,
$$\int_{\R^n} (M_{-1}f)^pw\,dx\leq C\int_{\R^n}|f|^pw\,.$$
\end{enumerate} 
\end{theorem}

Two weight norm inequalities are significantly more difficult to prove for the operators $M_{-1}$ and $M_0$.  One reason is that the covering techniques break down for the geometric and harmonic averages.  Because of this, the best results are one dimensional results.  Working on the real line, Cruz-Uribe, Neugebauer, and Olesen \cite{CNO} were able to prove the following theorem.

\begin{theorem}[\cite{CNO}] \label{thm:twhar} Given a pair of weights $(u,v)$ and $0<p<\infty$ let $\sigma=v^{\frac{1}{p+1}}$.  The following are equivalent
\begin{enumerate}[(i)]
\item the pair of weights satisfies
$$\dashint_I u\,\leq C \left(\,\dashint_I \sigma\,\right)^{p+1}$$
for all intervals $I$;
\item the operator satisfies the weak-type inequality
$$u(\{x\in \R: M_{-1}f(x)>\lambda\})\leq \frac{C}{\lambda^p}\int_{\R}|f|^pv\,;$$
\item the operator satisfies the strong-type inequality
$$\int_{\R} (M_{-1}f)^pu\,\leq C\int_{\R}|f|^pv\,;$$
\item the operator satisfies the testing condition
$$\int_I (M_{-1}(\sigma^{-1}\1_I))^pu\,\leq C\int_I\sigma\,$$
for all intervals $I$.  
\end{enumerate} 
\end{theorem}
A remarkable aspect of Theorem \ref{thm:twhar} is that the two weight $A_p$-like condition (formally an $A_{-p}$ condition) is sufficient for the strong type boundedness.  This is in stark contrast to the two weight results for the geometric and Hardy-Littlewood maximal functions.  

Yin and Muckenhoupt \cite{YM} studied two weight norm inequalities for $M_0$ proving the following one dimensional results.  Alternatively, Cruz-Uribe and Neugebauer \cite{CN2} were able to prove two weight norm inequalities for $M_0$ on the real line by approximating $M_0$ from below with the operators 
$$M_{-r}f(x)=\sup_{Q\ni x}\left(\,\dashint_Q |f|^{-r}\,\right)^{-\frac1r},$$
as $r\ra 0^+$.  By taking a limiting argument in Theorem \ref{thm:twhar} they were able to obtain the following results, again on the real line.

 \begin{theorem}[\cite{CN2},\cite{YM}] Suppose $(u,v)$ is a pair of weights defined on $\R$ and $0<p<\infty$.  Then the weak-type norm inequality 
$$u(\{x\in \R:M_0f(x)>\lambda\})\leq \frac{C}{\lambda^p}\int_\R |f|^pv\,$$
holds for $f\in L^p(v)$ if and only if the pair $(u,v)$ satisfies the two weight $A_\infty$ condition
$$\sup_I \left(\dashint_I u\,\right)\exp\left(-\dashint_I \log v\,\right)<\infty.$$  Moreover, the strong-type inequality
$$\int_{\R}(M_0f)^pu\,\leq C\int_{\R}|f|^pv\,$$
holds if and only if the testing condition
$$\int_I M_0(v^{-1}\1_I)u\,\leq C|I|$$
holds for all intervals $I$.  
\end{theorem}
We notice that the condition on the weights does not depend on $p$: this is to be expected since $M_0(f)^p=M_0(f^p)$ for $f\geq 0$.

When extending these results to higher dimensions or more general contexts one has to overcome serious difficulties.  One way to accomplish this is to assume doubling conditions on the weights.  A measure is doubling if $$\mu(2Q)\leq C\mu(Q)$$
for every cube $Q$ (here $2Q$ is the concentric cube with twice the sidelength of $Q$).   The smallest such $C$ will be called the doubling constant of $\mu$ and will be denoted $\mathsf{d}(\mu)$.  We can now state the higher dimensional results for $M_{-1}$ and $M_0$ both due to Cruz-Uribe \cite{CU}.

\begin{theorem}[\cite{CU}] Suppose $0<p<\infty$ and $(u,v)$ is a pair of weights such that either $u$ or $\sigma=v^{\frac{1}{p+1}}$ is a doubling weight.   The the following four conditions are equivalent
\begin{enumerate}[(i)]
\item the pair of weights satisfies
$$\dashint_Q u\,\leq C \left(\,\dashint_Q \sigma\,\right)^{p+1}$$
for all cubes $Q$ in $\R^n$;
\item the operator satisfies the weak-type inequality
$$u(\{x\in \R^n: M_{-1}f(x)>\lambda\})\leq \frac{C}{\lambda^p}\int_{\R^n}|f|^pv\,;$$
\item the operator satisfies the strong-type inequality
$$\int_{\R^n} (M_{-1}f)^pu\leq C\int_{\R^n}|f|^pv\,;$$
\item the operator satisfies the testing condition
$$\int_Q (M_{-1}(\sigma^{-1}\1_Q))^pu\leq C\int_Q\sigma\,$$
for all cubes $Q$.  
\end{enumerate} 
\end{theorem}

This situation is even worse for the geometric maximal operator.  In this case the weight $\sigma$ which depends on $p$ must have a bounded doubling constant as $p\ra \infty$ in order for the strong two weight norm inequalities to hold.  

\begin{theorem}[\cite{CU}] Suppose $0<p<\infty$ and $(u,v)$ is a pair of weights such that either the weight $u$ is doubling or the weight $\sigma_q=v^{\frac{1}{q+1}}$ is a doubling weight for all sufficiently large $q$ and
$$\limsup_{q\ra \infty} 2^{-nq}\mathsf{d}(\sigma_q)^{p+1}<\infty.$$
Then the weak-type inequality 
$$u(\{x\in \R^n:M_0f(x)>\lambda \})\leq \frac{C}{\lambda^p}\int_{\R^n}|f|^pv\,$$
holds for $f\in L^p(v)$ if and only if the pair of weights satisfies the two weight $A_\infty$ condition
$$\sup_Q \left(\dashint_Q u\,dx\right)\exp\left(-\dashint_Q \log v\,\right)<\infty$$
where the supremum is over all cubes in $\R^n$. 
\end{theorem}

\begin{theorem}[\cite{CU}] Suppose $0<p<\infty$ and $(u,v)$ is a pair of weights such that the weight $\sigma_q=v^{\frac{1}{q+1}}$ is a doubling weight for all sufficiently large $q$ and
$$\limsup_{q\ra \infty} \mathsf{d}(\sigma_q)<\infty.$$
Then the strong-type inequality 
$$\int_{\R^n}M_0f^pu\,\leq C\int_{\R^n}|f|^pv\,$$
holds for $f\in L^p(v)$ if and only if the pair of weights satisfies the testing condition
$$\int_Q (M_0(v^{-1}\1_Q))^pu\leq C|Q|$$
holds for all cubes in $\R^n$. 
\end{theorem}

\section{Preliminaries and Main Results}
We will study the harmonic and geometric operators with respect to a general basis of open sets.  By a basis we mean a collection $\mathcal{B}$ of bounded open sets in $\R^n$.  The most well known bases are the following:
\begin{enumerate}[(i)]
\item $\mathcal B=\mathsf Q$, the basis of all cubes with sides parallel to the axes;
\item $\mathcal B=\mathsf D$, the basis of all dyadic cubes from a fixed dyadic grid;
\item $\mathcal B=\mathsf R$, the basis of all rectangles. 
\end{enumerate}

We define Hardy-Littlewood maximal operator with respect to a general basis as
$$M^\B f(x)=\sup_{\substack{B\in \B \\ x\in B}}\,\dashint_B |f|\,.$$
The classes $A^\B_p$ and $A^\B_\infty$ will denote the $A_p$ and $A_\infty$ classes with respect the basis $\mathcal B$.  They are defined similarly to  \eqref{Ap} and \eqref{Ainfty} except with the supremum over all sets from the basis $\B$ instead of the $\mathsf Q$.  We say that a weight satisfies condition $\mathsf A$ if there exists constants $0<\al<1$ and $c=c(\al)$ such that
\begin{equation}\label{condA} w(\{x\in \R^n:M^\B(\1_E)(x)>\al\})\leq c\,w(E)\end{equation}
for all measurable sets $E$. Condition \textsf{A} was introduced in \cite{P2} and can be thought of as a restricted weak-type inequality.   It was believed to be weaker than $A^\B_\infty$, however, recently in \cite{HP} (see also \cite{HP2} and \cite{HLP}) it is shown that condition \textsf{A} is equivalent to $A_\infty^\B$ for several bases such as $\mathsf R$.   We also refer readers to the manuscript \cite{DMO} for other equivalent definitions of $A_\infty^\B$.  Finally, we say that $\B$ is a Muckenhoupt basis if for each $p$, $1<p<\infty$ and every $w\in A_p^\B$, $M^\B$ is bounded on $L^p(w)$.  P\'erez \cite{P2} proved that $\B$ is a Muckenhoupt basis if and only if for each $p$, $1<p<\infty$, and $w\in A_\infty^\B$ the weighted maximal operator
$$M^\B_w f(x)=\sup_{\substack{B\in \B \\ x\in B}}\frac{1}{w(B)}\int_B|f|w\,,$$
is bounded on $L^p(w)$. 

The study of $M^\B$ goes back to Zygmund who proved bounds for the basis of rectangles. Jawerth \cite{J} gave a systematic study of the one weight and two weight inequalities for the Hardy-Littlewood maximal operator with respect to a general basis.  We will use the convention to tuck the weight into the operator.  Namely the inequality
$$\int_{\R^n} (M^\B f)^pu\,\leq C\int_{\R^n} |f|^pv\,dx$$
is equivalent to the inequality 
\begin{equation}\label{tucked} \int_{\R^n} M^\B(f\sigma)^pu\,\leq C\int_{\R^n} |f|^p\sigma\,\end{equation}
where $\sigma=v^{1-p'}$.  The advantage is that the latter inequality make sense for general measures $\sigma$.  Notice that inequality \eqref{tucked} is equivalent to the maximal operator $M^\B(\sigma\,\cdot\,)$ being bounded from $L^p(\sigma)$ to $L^p(u)$.

\begin{theorem}[\cite{J}] Suppose $1<p<\infty$, $\B$ is a basis, and $(u,\sigma)$ is a pair of weights such that $M_\sigma^\B:L^p(\sigma)\ra L^p(\sigma)$.   Then the following inequality
$$\int_{\R^n} M^\B(f\sigma)^pu\,\leq C\int_{\R^n} |f|^p\sigma\,dx $$
holds for all $f\in L^p(\sigma)$ if and only if the testing condition 
$$\int_FM^\B(\1_F\sigma)^p\,u\leq C\sigma(F)$$
holds for all finite unions $F$ of sets in $\B$. 
\end{theorem}

In general the two weight $A_p$ condition 
$$\sup_{B\in \B}\left(\dashint_B u\,\right)^{1/p}\left(\dashint_B \sigma\,\right)^{1/p'}<\infty,$$
is necessary but not sufficient for the boundedness $M^\B:L^p(v)\ra L^p(u)$.  P\'erez showed a stronger condition, one made by bumping up the average on $\sigma$, is sufficient for the two weight boundedness of $M^\B$.  P\'erez \cite{P1} proved the following two weight bump result in the vein of this paper. 
\begin{theorem}[\cite{P1}] Suppose $1<p<\infty$ and $\B$ is a basis such that $M^\B$ is bounded on $L^s(\R^n)$ for all $1<s<\infty$.  If $(u,\sigma)$ is a pair of weights such that $u$ satisfies condition $\mathsf A$ and there exists $r>1$ such that the bumped $A_p$ condition
$$\sup_{B\in \B} \left(\,\dashint_B u\,\right)^{1/p}\left(\,\dashint_B \sigma^r\,\right)^{1/rp'}<\infty $$
holds, then the inequality
$$\int_{\R^n}M^\B(f\sigma)^pu\,dx\leq C\int_{\R^n} |f|^p\sigma\,$$
holds for all $f\in L^p(\sigma)$. 
\end{theorem}

Given a basis $\B$, define the harmonic and geometric maximal operators as
$$M^{\B}_{-1}f(x)=\sup_{\substack{B\in \B \\ x\in B}}\left(\,\dashint_B |f|^{-1}\,\right)^{-1}$$
and
$$M^{\B}_{0}f(x)=\sup_{\substack{B\in \B \\ x\in B}}\exp\left(\,\dashint_B \log|f|\,\right)$$
respectively.  Again we use the convention that $1/0=\infty$ and we define $M^{\B}_{-1}f(x)=M^{\B}_0f(x)=0$ if $x\notin \bigcup_{B\in \B}B$.   Given a weight, $w$, define the harmonic maximal operator with respect to $\B$ and $w$ by
$$M_{-1,w}^\B f(x)=\sup_{\substack{B\in \B \\ x\in B}} \left(\,\frac{1}{w(B)}\int_B |f|^{-1}w\,\right)^{-1}.$$

We are now ready to state our main results.  Again we notice that the inequality
$$\int_{\R^n} (M_{-1}^\B f)^pu\,\leq C\int_{\R^n} |f|^pv\,$$
is equivalent to the inequality
\begin{equation}\label{M-1}\int_{\R^n} M_{-1}^\B(f\sigma^{-1})(x)^pu\,\leq C\int_{\R^n} |f|^p\sigma\,\end{equation}
when $\sigma=v^{\frac{1}{p+1}}$.  Inequality \eqref{M-1} says that the operator $M_{-1}^\B(\sigma^{-1}\,\cdot\,)$ is bounded from $L^p(\sigma)$ to $L^p(u)$.  Our first result is a testing characterization that parallels earlier results in \cite{J}.  

\begin{theorem} \label{testingHM} Suppose  $(u,\sigma)$ is a pair of weights and $p$ is an exponent with $0<p<\infty$.  Suppose further that $M_{-1,\sigma}^\B$ is bounded on $L^p(\sigma)$.  Then 
$$M^\B_{-1}(\sigma^{-1}\,\cdot\,):L^p(\sigma)\ra L^p(u)$$
if and only if there exists a constant $C$ such that
$$\int_FM_{-1}^\B(\sigma^{-1}\1_F)^pu\,\leq C \sigma(F),$$
for all $F$ such that $F$ is a finite union of sets in $\B$.  
\end{theorem}

The assumption $M_{-1,\sigma}^\B$ be bounded on $L^p(\sigma)$ is not a strong assumption.  Indeed $M_{-1,\sigma}^\B f$ is bounded by $M_\sigma^\B(|f|^r)^{1/r}$ for any $r>0$.  Thus the hypothesis of Theorem \ref{testingHM} will be satisfied if the maximal operator $M^\B_\sigma$ is bounded on $L^p(\sigma)$ for large $p$. In particular if $\sigma$ belongs to $A_\infty$ and $\B$ is a Muckenhoupt basis then the assumption is satisfied (P\'erez \cite{P2}).

Our next result is a sufficient bump condition for the harmonic maximal operator.   
\begin{theorem} \label{bumpHM} Suppose that $0<p<\infty$, $(u,\sigma)$ is a pair of weights, and $\B$ is a basis such that $M^\B$ is bounded on $L^s(\R^n)$ for {\color{red}} $1<s<\infty$.  If $u$ satisfies condition $\mathsf{A}$ and there exists $r$, $0<r<1$ such that $(u,\sigma)$ satisfies
$$\left(\,\dashint_B u\,\right)\leq C \left(\,\dashint_B \sigma^r\,\right)^{\frac{p+1}{r}}$$
for all $B\in \B$ and some constant $C$, then
$$\int_{\R^n}M_{-1}^\B(f\sigma^{-1})^pu\,\leq C\int_{\R^n}|f|^p\sigma\,$$
for all $f\in L^p(\sigma)$.

\end{theorem}

The bump condition for $M^\B_{-1}$ requires a power $r<1$ instead of $r>1$.  This is due to the nature of the weighted constant on $(u,\sigma)$ with $\sigma$ being on the right side of the inequality.

It is unclear how to extend these results to the geometric maximal operator.  One obstacle of extending Theorem \ref{testingHM} is the fact that it uses the boundedness of the weighted harmonic maximal operator $M_{-1,\sigma}^\B$.  We remark that we do not know how to extend this result to the geometric maximal operator because it is unclear how to take a limit of the bump condition.

We do show $M_0^\B$ can be approximated from below with the operators
$$M_{-r}^\B f= M_{-1}^\B(|f|^r)^{\frac1r}, \qquad r>0.$$
In fact, if we define 
$$M_{0^-}^\B f=\lim_{r\ra 0^+}M_{-r}^\B f$$
then we have the following lemma.

\begin{lemma} \label{approx} Suppose $f$ is a non-negative measurable function on a fixed cube $Q_0$, possibly of infinite measure, such that $f^{-1}$ belongs to $L^r_{\text{loc}}(Q_0)$ for some $r>0$.  Then for all $x\in \R^n$
$$M_{0^-}^\B(f\1_{Q_0})(x)=M_0^\B(f\1_{Q_0})(x).$$
\end{lemma}

The plan of the paper is as follows.  In Section \ref{sect3} we will prove the two weight testing characterizations,  Theorems \ref{testingHM} and Lemma \ref{approx}.  In Section \ref{sect4} we will prove Theorem \ref{bumpHM}.  We will end with Section \ref{sect5} and some observations for the basis of dyadic cubes.

\section{Two weight testing conditions}\label{sect3}
Our proof of Theorem \ref{testingHM} will follow the original proof of Jawerth for the maximal operator associated to $\B$, which uses a discretization of the operator $M_{-1}^\B$.  

\begin{proof}[Proof of Theorem \ref{testingHM}]  For the moment we will assume that $f$ is a non-negative function, supported on a fixed cube $Q_0$, and is bounded above on that cube.  We will also momentarily assume that $\sigma$ is bounded below.   These assumption ensure that averages of the form
$$ \dashint_B f^{-1}\sigma $$
are always non zero if $B\in \B$ and satisfies $B\subset Q_0$.  We will remove these restrictions at the end of the proof.  First notice that if $\lambda>0$ and $M_{-1}^\B(f\sigma^{-1})(x)>\lambda$ then there exists $B\in \B$ such $B\subset Q_0$ and
\begin{equation}\label{level}\left(\,\dashint_{B} f^{-1}\sigma\right)^{-1}>\lambda.\end{equation}
Indeed, by the definition of the $M_{-1}^\B(f\sigma^{-1})(x)$ there exists $B\in \B$ that satisfies \eqref{level}.  Moreover, $B\subset Q_0$ because if not then
$$\left(\dashint_B f^{-1}\sigma\right)^{-1}=|B|\left(\int_{B\cap Q_0} f^{-1}\sigma+\int_{B\backslash Q_0}f^{-1}\sigma\right)^{-1}=0.$$

Let 
$$\Omega_k=\{x\in Q_0: 2^k<M^\B_{-1}(f\sigma^{-1})(x)\leq 2^{k+1}\}.$$
From the definition of $M^\B_{-1}$ we have that if $\Omega^k\not=\varnothing$ then $\Omega^k\subset \bigcup_{j} B^k_j$ where $B^k_j\in \B$, $B^k_j\subset Q_0$, and satisfying
$$\left(\,\dashint_{B^k_j} f^{-1}\sigma\right)^{-1}>2^k.$$
Set $E^k_1=B_1^k\cap \Omega_k$ and for $j>1$ set 
$$E^k_j=\Big(B^k_j\backslash \bigcup_{i=1}^{j-1} B^k_i\Big)\cap \Omega_k.$$
Then the sets $\{E^k_j\}_{j,k}$ are pairwise disjoint and $\Omega_k=\bigcup_j E^k_j.$
We are now ready to estimate $\|M_{-1}^\B(f\sigma^{-1})\|_{L^p(u)}$.  We have
\begin{align*}
\int_{\R^n}M_{-1}^\B(f\sigma^{-1})^p u &= \sum_k \int_{\Omega_k} M^\B_{-1}(f\sigma^{-1})^pu \\
&\leq 2^p\sum_{j,k} 2^{kp}u(E^k_j)\\
&\leq 2^p\sum_{j,k} \left(\,\dashint_{B^k_j} f^{-1}\sigma\right)^{-p}u(E^k_j)\\
&=2^p\sum_{j,k} \left(\frac{1}{\sigma(B^k_j)}\int_{B^k_j} f^{-1}\sigma\right)^{-p}u(E^k_j)\left(\frac{|B^k_j|}{\sigma(B^k_j)}\right)^p
\end{align*}
On the measure space $\mathcal X=\N\times \Z$ define the function 
$$F(j,k)=\left(\frac{1}{\sigma(B^k_j)}\int_{B^k_j} f^{-1}\sigma\right)^{-p}$$
and the measure
$$\mu(j,k)=u(E^k_j)\left(\frac{|B^k_j|}{\sigma(B^k_j)}\right)^p.$$
Then we have
$$\int_{\R^n}M_{-1}^\B(f\sigma^{-1})^p u\leq 2^p\int_\X F\,d\mu=2^p\int_0^\infty \mu(\{(j,k)\in \X: F(j,k)>\lambda\})d\lambda.$$
Given $\lambda>0$ and $N\in \N$ set 
$$\Gamma_N(\lambda)=\{(j,k)\in \X: j+|k|\leq N, F(j,k)>\lambda\} \ \ \text{and} \ \ G_N(\lambda)=\bigcup_{(j,k)\in \Gamma_N(\lambda)}B^k_j,$$
so that $\Gamma_N(\lambda)$ is a finite union of sets in $\B$.  Then by the testing condition
\begin{align*}
\mu(\Gamma_N(\lambda))&=\sum_{(j,k)\in \Gamma_N(\lambda)} u(E^k_j)\left(\frac{|B^k_j|}{\sigma(B^k_j)}\right)^p\\
&\leq \sum_{(j,k)\in \Gamma(\lambda)}\int_{E^k_j} M^\B_{-1}(\1_{G_N(\lambda)}\sigma^{-1})^pu\\
&\leq \int_{G(\lambda)}  M^\B_{-1}(\1_{G_N(\lambda)}\sigma^{-1})^pu\\
&\leq C\sigma(G_N(\lambda)).
\end{align*}
Moreover, if $x\in G_N(\lambda)$ then $x\in B^k_j$ for some $j$ and $k$ with 
$$\lambda<\left(\dashint_{B_j^k} f^{-1}\sigma\right)^{-p}\leq M_{-1,\sigma}^\B f(x)^p,$$
which is to say that 
$$G_N(\lambda)\subset \{x:(M_{-1,\sigma}^\B f)^p>\lambda\}.$$
Letting $N\ra \infty$ we have 
$$\mu(\{(j,k)\in \X:F(j,k)>\lambda\})\leq C\sigma(\{x:(M_{-1,\sigma}^\B f)^p>\lambda\}).$$
Combining this calculation with the previous estimates we have

\begin{multline*}\int_{\R^n}M_{-1}^\B(f\sigma^{-1})^pu\leq C\int_0^\infty\sigma(\{x:M_{-1}^\B(f\sigma^{-1})^p>\lambda\})d\lambda\\
=C\int_{\R^n}(M_{-1,\sigma}^\B f)^p\sigma\leq C\int_{\R^n} f^p\sigma\end{multline*}
where we used the assumption $M_{-1,\sigma}^\B:L^p(\sigma)\ra L^p(\sigma)$.  To remove the assumptions on $f$ and $\sigma$ assume that $f\in L^p(\sigma)$ and $f\geq 0$.  Notice that the inequality 
$$\int_{\R^n}M_{-1}^\B(f\sigma^{-1})^pu\leq C\int_{\R^n} f^p\sigma, \qquad f\geq 0$$
is equivalent to 
$$\int_{\R^n}(M_{-1}^\B f)^pu\leq C\int_{\R^n} f^pv, \qquad f\geq 0$$
where $v=\sigma^{p+1}$.  Since $v>0$ and $f$ is bounded and supported on a cube we have
$$\int_{\R^n} (M_{-1}^\B f)^pu\leq C\int_{\R^n} f^p v.$$
Given $N\in \N$ let $Q_N=[-N,N]^n$ and let
$$f_N=\Big(\frac{1}{f}+\frac{1}{N}\Big)^{-1}{\1_{Q_N}}$$
and 
$$v_N=\sigma^{p+1}+N^{-p-n-1}.$$
Then, if $\sigma_N=v_N^{1/(p+1)}$ we have $\sigma_N\geq \sigma$ and $\sigma_N$ is bounded below.   Given any finite union of sets in our basis, $F$, we have 
$$\int_FM_{-1}^\B(\1_F\sigma_N^{-1})^p u\leq \int_FM_{-1}^\B(\1_F\sigma^{-1})^p u\leq C\int_F\sigma\leq C\int_F \sigma_N.$$
In particular, $\sigma_N$ satisfies the testing condition with the same constant as $\sigma$.  We now make some observations about $f_N$.  First, clearly $f_N$ is supported on the cube cube $Q_N$.  Second, $f_N\leq \min(f,N)$ so $f_N$ is bounded above.   Finally, the sequence $f_N$ is increasing since it is zero off $Q_N$, and on $Q_N$ we have
$$\frac{1}{f_{N+1}}=\frac{1}{f}+\frac{1}{N+1}\leq \frac{1}{f}+\frac{1}{N}=\frac{1}{f_N}.$$
Then $f_N$ and $\sigma_N$ satisfies the restricted hypothesis at the beginning of the proof, so we have for $v=\sigma^{p+1}$
\begin{multline*}\int_{\R^n}M_{-1}^\B(f_N)^pu\leq C\int_{Q_N}f_N^pv_N\leq C\int_{\R^n} f^pv\\
+C\int_{Q_N} N^pN^{-p-n-1}
\leq C\int_{\R^n}f^pv+\frac{C}{N}.\end{multline*}
Since $f_N$ is an increasing sequence we also have that $M_{-1}^\B(f_N)$ is an increasing sequence and since $f_N\leq f$ we have
$$\lim_{N}M_{-1}^\B(f_N)\leq M_{-1}^\B(f).$$
On the other hand let $\ep>0$ and $x\in \R^n$.  Then there exists $B\in \B$ such that $x\in B$ and
$$M_{-1}^\B f(x)-\ep<\left(\dashint_B f^{-1}\right)^{-1}.$$
If $\dashint_B f^{-1} =\infty$ then
$$M_{-1}^\B(f)(x)-\ep\leq 0 \leq  M_{-1}^\B(f_N)(x).$$
Otherwise, $f>0$ on $B$ and since $B$ is bounded we have that $B\subset Q_N$ for $N$ large and 
$$\dashint_B \frac1f =\dashint_B \frac1{f_N}-\frac1N\geq \Big(\,\inf_{B\ni x} \dashint_B \frac1{f_N}\,\Big)-\frac1N=[M_{-1}^\B(f_N)(x)]^{-1}-\frac{1}{N}.$$
Letting $N\ra \infty$ we have
$$\dashint_B \frac{1}{f}\geq \lim_N[M_{-1}^\B(f_N)(x)]^{-1}.$$
Then
$$M_{-1}^\B f(x)-\ep\leq \lim_NM_{-1}^\B(f_N)(x),$$
and since $\ep>0$ we have that the sequence $M_{-1}^\B(f_N)$ increases to $M_{-1}^\B f$.  By the monotone convergence theorem
$$\int_{\R_n}(M_{-1}^\B f)^pu\leq C\int_{\R^n}f^pv$$
with $v=\sigma^{p+1}$, which is equivalent to the desired inequality.
 \end{proof}

\begin{proof}[Proof of Lemma \ref{approx}] By Jensen's inequality we have 
$$\lim_{r\ra 0^+}M_{-r}^\B(f\1_{Q_0})(x)\leq M_0^\B(f\1_{Q_0})(x).$$
On the other hand if $x\notin Q_0$ then 
 $$\lim_{r\ra 0^+}M_{-r}^\B(f\1_{Q_0})(x)=M_0^\B(f\1_{Q_0})(x)=0.$$ 
 Let $x\in Q_0$ and $\ep>0$, then we may assume that there exists $B\in \B$ with $x\in B$ such that $B\subset Q_0$ and
 $$M_0(f\1_{Q_0})(x)-\ep<\exp\left(\dashint_B \log |f|\right).$$
 If no such $B$ exists then again both $M_0^\B f(x)$ and $\lim_{r\ra 0^+}M_{-r}^\B f(x)$ are zero.  Now we have
 \begin{multline*}
 M_0(f\1_{Q_0})(x)-\ep<\exp\left(\dashint_B \log |f|\right)=\left[\exp\left(\dashint_B \log|f|^{-1}\right)\right]^{-1}\\
 =\left[\lim_{r\ra 0^+}\left(\dashint_B |f|^{-r}\right)^{\frac1r}\right]^{-1}=\lim_{r\ra 0^+}\left(\dashint_B |f|^{-r}\right)^{-\frac1r}\leq \lim_{r\ra 0^+}M_{-r}^\B(\1_{Q_0} f)(x).
 \end{multline*}

\end{proof}

\section{Two weight bump conditions}\label{sect4}

We would like to use the same techniques in Theorem \ref{testingHM} to prove Theorem \ref{bumpHM}.  However, one of the main difficulties is that we have no control over the size of the disjoint sets $E^k_j$.  It is here that we use condition $\mathsf{A}$ on the weight $u$ (see inequality \eqref{condA}).  We begin with a lemma whose proof can be found in \cite{GLPT}. 
 
 \begin{lemma}\label{lemma}
Let $\B$ be a basis and $w$ a weight associated to this basis. Suppose further that $w$ satisfies condition \textsf{A} with constants $0<\al<1$ and $c=c(\al)$. Then given any finite sequence $\{ A_i \}_{i=1}^M$ of sets $ \B$, we can find a subsequence $\{\tilde{ A_i} \}_{i \in I}$ of  $\{ A_i \}_{i=1}^M$ such that the following hold: for each $1\leq i<j\leq M$ we have 
\begin{enumerate}[(i)]
\item for each $i\in I$
$$ \Big|\tilde{A}_i\cap \bigcup_{\substack{s\in I\\ s<i}}\tilde{A}_s\Big|\leq \al|\tilde{A}_i|,$$
\item for each $1\leq i<j\leq M+1$
$$u\Big(\bigcup_{1\leq s<j}A_s\Big)\leq c\left[u\Big(\bigcup_{1\leq s<i}A_i\Big)+u\Big(\bigcup_{\substack{s\in I\\ i\leq s<j}} \tilde{A_s}\Big)\right]$$
 \end{enumerate}

\end{lemma}
 



\vspace{3mm}










We are now ready to prove Theorem \ref{bumpHM}.

\begin{proof}[Proof of Theorem \ref{bumpHM}]
We will assume again that $f$ is supported on a cube and that $f$ is a bounded function on that cube and that $\sigma$ is bounded below.  The limiting argument presented in the proof of Theorem \ref{testingHM} will allow us to pass to general $f\in L^p(\sigma)$.  Since $f$ is a bounded function with compact support we have that $M_{-1}^\B(f\sigma^{-1})$ is bounded and hence finite a.e.  Fix $N\in \N$, we shall estimate

$$ \int_{\{x:2^{-N} < M^{\B}_{-1}f(x) \leq 2^{N+1}\}} M^{\B}_{-1} (f \s^{-1})^p u.$$
Our estimates will not depend on $N$ so a limiting argument will allow us to obtain all of $\R^n$.

\vspace{3mm}

\noindent For each $k\in \Z$ with $|k| \le N$, we can find a compact 
$$K_k \subseteq  \{ x \in \R^n : M^{\B}_{-1}(f\s^{-1})(x)  >2^k \}$$
and  
$$u( \{ x \in \R^n : M^{\B}_{-1}(f\s^{-1})(x)  >2^k \}) \leq 2u(K_k).$$

We will now use a selection process from \cite{GLPT} (see also \cite{J} and \cite{LL}).  In \cite{GLPT} the selection process was carried out for the basis $\mathsf{R}$ but the same procedure works for a general basis.  We repeat the details here for the convenience of the the reader.  For each $|k|\leq N$ there exists a finite collection of sets in $\B$, $\{B_j^k\}_j$ that cover $K_k$ and satisfy 
$$\left(\avgBjko f^{-1}\s \right)^{-1}> 2^k .$$
For convenience, we set $b_k =\{B_j^k\}_j$ if $|k|\leq N$ and  $b_k=\varnothing$ if $|k| > N$.  Also set and
$$\W_k = \begin{cases} \bigcup_{s \geq k} \bigcup_j B_j^s &\mbox{when } |k| \leq N \\ 
\varnothing & \mbox{when }|k| > N. \end{cases} $$
Observe that these sets are decreasing in $k$, i.e., $\W_{k+1} \subset \W_k$.  We will now rearrange the sets in the $b_k$'s into a double indexed sequence $\{A_i(l)\}_{i\geq 1,1\leq l\leq \mu}$ where $\mu$ is a large number to be chosen later.  Set $i_0(0)=1$.  Let $i_1(0)-1$ be the number of sets in $b_N=\{B_j^N\}_{j}$ and define
$$A_i(0)=B_i^N, \qquad i_0(0)=1\leq i<i_1(0).$$
Next, let $i_2(0)-i_1(0)$ be the number of sets in $b_{N-\mu}=\{B^{N-\mu}_j\}_j$ and set
$$A_i(0)=B_i^{N-\mu}, \qquad i_1(0)=1\leq i<i_2(0).$$
Continue this process we reach the first integer $m_0$ such that $N-(m_0+1)\mu< -N$.  At this point we let 
$$A_i(0)=B_i^{N-m_0\mu}, \qquad i_{m_0}\leq i<i_{m_0+1}(0).$$
Define the sequence $\{A_i(1)\}_{i}$ to be first the sets of $b_{N-1}=\{B_j^{N-1}\}_{j}$ followed by the sets of $b_{N-1-\mu}$ and continue until the first integer $m_1$ such that $N-1-(m_1+1)\mu<-N.$  Finally, continue this process until the sets of all of the $b_k$'s are exhausted.

Since $u$ satisfies condition \textsf{A} we can apply Lemma \ref{lemma} to each $\{A_i (l) \}_{i\geq 1}$ for a fixed $\al$ to obtain sequences
	$$\{\tilde{A_i} (l) \}_{i\geq 1} \subset \{A_i (l) \}_{i\geq 1}, \hspace{5mm} 0\leq l \leq \m-1,$$
From the definition of the set $\W_k$ and the construction of the families $\{A_i (l) \}_{i\geq 1}$, we can use Lemma \ref{lemma} to obtain
	\begin{multline*}
	u(\W_k) \leq c\left[ u(\W_{k+\m}) + u\left( \bigcup_{i_{m_l} (l) \leq i < i_{m_l +1} (l)} \tilde{A_i} (l) \right) \right] \\
	\leq c\,u(\W_{k+\m}) + c\sum_{i=i_{m_l}(l)}^{i_{m_l +1}(l) -1} u(\tilde{A_i}(l))
	\end{multline*}
if $k = N- l - m\m$. It suffices to consider these indices $k$ because the sets $\W_k$ are decreasing.

The sets $\{ \tilde{A_i} (l) \}_{i=i_{m_l}(l)}^{i_{m_l +1}(l) -1}$ belong to $b_k$ with $k = N - l - m\m$ and therefore
	$$\left(  \dashint_{\tilde{A_i}(l)} f^{-1}\s \right)^{-1} > 2^k . $$ 
By Lemma \ref{lemma} we  have
\begin{multline*}
\int_{\{2^{-N}<M^{\B}_{-1}f \leq 2^{N+1}\}} M^{\B}_{-1} (f \s^{-1})^p u \lesssim  \sum_k 2^{kp}u(\Omega_{k})\\
\lesssim \sum_k 2^{kp}u(\Omega_{k+\mu})+\sum_{l=0}^{\mu-1}\sum_{i=i_m(l)}^{i_{m+1}(l)-1}u(\tilde{A_i}(l))\left(  \dashint_{\tilde{A_i}(l)} f^{-1}\s \right)^{-p}. 
\end{multline*}
Since the sum $\sum_k 2^{kp}u(\Omega_k)$ is finite and $\sum_k 2^{kp}u(\Omega_{k+\mu})\leq 2^{-p\mu}\sum_{k}2^{kp}u(\Omega_k)$ we may choose $\mu$ large enough to ignore the first summation.  For the other term we have
\begin{multline*}
\sum_{l=0}^{\mu-1}\sum_{i=i_m(l)}^{i_{m+1}(l)-1}u(\tilde{A_i}(l))\left( \frac{1}{|\tilde{A_i}(l)|} \int_{\tilde{A_i}(l)} f^{-1}\s \right)^{-p}\\
\lesssim  \sum_{l,i} \left( \dashint_{\tilde{A_i}(l)} f^{-1}\s \right)^{-p}|\tilde{A_i}(l)| \left( \dashint_{\tilde{A_i}(l)}\s^r \right)^{\frac{p+1}{r}}.
\end{multline*}
Consider $\left(\dashint_B \s^r\right)^{\frac{p+1}{r}}$ for a general $B\in \B$. Using Holder's inequality with 
$$s=\frac{p+1}{rp} \quad \text{and} \quad s'=\frac{p+1}{p+1-rp}$$ 
we find that 
\begin{align}\left(\avgB \s^r\right)^{\frac{p+1}{r}}& = \left(\avgB \s^r (f\s^{-1})^{\frac{1}{s}} (f\s^{-1})^{-\frac{1}{s}}  \right)^{\frac{p+1}{r}}\nonumber\\
& \leq \left(\avgB  \left(\s^r (f\s^{-1})^{\frac{1}{s}}\right)^{s'}\right)^{\frac{p+1}{rs'}} \left( \avgB  (f\s^{-1})^{-\frac{s}{s}}\right)^{\frac{p+1}{rs}}\nonumber\\
&= \left( \avgB (f^{\frac{s'}{s}}\sigma^{rs'-\frac{s'}{s}}) \right)^{\frac{p+1}{rs'}} \left( \avgB f^{-1} \s \right)^p\nonumber\\
\label{holder}&= \left( \dashint_B (f^{p}\sigma)^{\frac{rs'}{p+1}} \right)^{\frac{p+1}{rs'}} \left( \avgB f^{-1} \s \right)^p
\end{align}
where we have used the calculations
$$\frac{s'}{s}=s'\frac{rp}{p+1}=p\frac{rs'}{p+1}, \quad \text{and} \quad rs'-\frac{s'}{s}=s'\Big(r-\frac1s\Big)=\frac{rs'}{p+1}.$$
Letting 
$$t=\frac{p+1}{rs'}=\frac{p+1-rp}{r}>1$$
and using inequality \eqref{holder} we obtain
\begin{align*}
\lefteqn{\int_{\{2^{-N}<M^{\B}_{-1}f \leq 2^{N+1}\}} M^{\B}_{-1} (f \s^{-1})^p u}\\ 
& \lesssim  \sum_{l=0}^{\mu-1}\sum_{i=i_m(l)}^{i_{m+1}(l)-1} \left( \dashint_{\tilde{A_i}(l)} f^{-1}\s \right)^{-p}|\tilde{A_i}(l)| \left( \dashint_{\tilde{A_i}(l)} \s^r \right)^{\frac{p+1}{r}}\\
&\lesssim\sum_{l=0}^{\mu-1}\sum_{i=i_m(l)}^{i_{m+1}(l)-1}\left( \dashint_{\tilde{A_i}(l)}(f^p \s)^{\frac{1}{t}} \right)^{t} |\tilde{A_i}(l)|.
\end{align*}
For each $l$ let $E_1(l)=\tilde{A_i}(l)$ and $E_i(l)=\tilde{A_i}(l)\backslash \bigcup_{s<i}\tilde{A_s}(l)$ for $i>1$.  Then the sets $\{E_i(l)\}$ are pairwise disjoint and using property (i) of Lemma \ref{lemma} we have that $|A_i(l)|\leq c|E_i(l)|$.  Continuing with the estimates we have
\begin{align*}
\sum_{l=0}^{\mu-1}\sum_{i=i_m(l)}^{i_{m+1}(l)-1}\left( \dashint_{\tilde{A_i}(l)}(f^p \s)^{\frac{1}{t}} \right)^{t} |\tilde{A_i}(l)|&\lesssim \sum_{l=0}^{\mu-1}\sum_{i=i_m(l)}^{i_{m+1}(l)-1}\left( \dashint_{\tilde{A_i}(l)}(f^p \s)^{\frac{1}{t}} \right)^{t} |E_i(l)|\\
&\leq\sum_{l=0}^{\mu-1}\sum_{i=i_m(l)}^{i_{m+1}(l)-1}\int_{E_i(l)}M^\B\big((f^p\sigma)^{\frac1t}\big)^t \\
&\leq\int_{\R^n}M^\B\big((f^p\sigma)^{\frac1t}\big)^t\lesssim \int_{\R^n} f^p\sigma
\end{align*}
where we used that the sets $\{E_i(l)\}$ are pairwise disjoint and the maximal function $M^\B$ is bounded on $L^t$ for $t>1$.  This completes the proof of Theorem \ref{bumpHM}.
\end{proof}

 \section{Dyadic grids} \label{sect5}

We consider the specific case of our maximal operators working over a general dyadic grid $\mathsf D$.  A dyadic grid is a collection of cubes that satisfy the following properties:
\begin{itemize}
\item if $Q\in \D$ then $\ell(Q)=2^k$ for some $k\in \Z$;
\item if $Q,P\in \D$, then $Q\cap P\in \{\varnothing,Q,P\}$;
\item for each fixed $k\in\Z$ the set $\D_k=\{Q\in \D:\ell(Q)=2^k\}$ is a partition of $\R^n$.
\end{itemize}
The standard dyadic grid consists of cubes $Q$, open on the right, whose vertices are adjacent points of the lattice $(2^{-k}\Z)^{n}$.  Technically, a dyadic grid is not a basis since its members are not open sets.  However, we will treat the dyadic grid $\D$ as a basis, since the boundary of a cube has measure zero.  Given a dyadic grid $\D$ we define our respective operators accordingly:

$$M^{\D}_{-1}f(x)=\sup_{\substack{Q\in \mathsf D \\ x\in Q}}\left(\,\dashint_Q |f|^{-1}\,\right)^{-1}$$
and
$$M^{\D}_{0}f(x)=\sup_{\substack{Q\in \mathsf D \\ x\in Q}}\exp\left(\,\dashint_Q \log|f|\,\right).$$

It was stated in \cite[Section 1.4]{CU} that the doubling assumptions may be removed in higher dimensions if the harmonic and geometric maximal operators are changed to dyadic versions.  In \cite{CU} it is left to the reader to complete the details.  We now provide the details for the results in \cite{CU} for the dyadic harmonic and geometric maximal operators in higher dimensions without doubling assumptions on the weights.  Previously, the only known higher dimensional results that did not require doubling assumptions on the weights were for the centered harmonic operator \cite[Theorem 1.7]{CU}.  
\begin{theorem} \label{dyadic results}
Let $p$ be an exponent satisfying $0<p<\infty$ and $(u,\sigma)$ be a pair of weights. Then the following are equivalent 
\begin{enumerate}[(i)]
	\item the pair of weights $(u,\sigma)$ satisfies
		$$\dashint_Q u\,\leq C \left(\,\dashint_Q \sigma\,\right)^{p+1}$$
for all cubes $Q \in \mathsf D$;
	\item the operator $M^{\D}_{-1}$ satisfies the weak-type inequality
		$$u(\{x\in \R^n: M^\D_{-1}(f\sigma^{-1})(x)>\lambda\})\leq \frac{C}{\lambda^p}\int_{\R^n}|f|^p\sigma\,;$$
	\item the operator $M^{\D}_{-1}$ satisfies the strong-type inequality
		$$\int_{\R^n} M^\D_{-1}(f\sigma^{-1})^pu\,\leq C\int_{\R^n}|f|^p\sigma\,;$$
	\item the operator $M^{\D}_{-1}$ satisfies the testing condition
		$$\int_Q (M^\D_{-1}(\1_Q\sigma^{-1}))^pu\,\leq C\int_Q\sigma\,$$
for all cubes $Q \in \mathsf D$.  
\end{enumerate}
  
\end{theorem}
The removal of the doubling condition relies wholly on a specific geometric property of the cubes in $\mathsf D$. The property is that any two cubes in $\mathsf D$ are either nested or disjoint. This well-known property allows us to use the universal maximal operators with respect to a weight $\sigma$:
$$M^\D_{-1,\sigma}f(x)=\sup_{\substack{Q\in \D \\ x\in Q}}\left(\frac{1}{\sigma(Q)}\int_Q |f|^{-1}\sigma\right)^{-1}$$
and
$$M^\D_{0,\sigma}f(x)=\sup_{\substack{Q\in \D \\ x\in Q}}\exp\left(\frac{1}{\sigma(Q)}\int_Q (\log|f|) \sigma\right).$$
Finally we introduce one more limiting operator:
$$M^\D_{0^+,\sigma}f=\lim_{r\ra0^+} M^\D_{r,\sigma}f=\lim_{r\ra 0}M_\sigma^\D(|f|^r)^\frac1r.$$
It is clear that for any power $r>0$ we have
$$M_{-r,\sigma}^\D f\leq M_{0^-,\sigma}^\D f\leq M_{0,\sigma}^\D f\leq M_{0^+,\sigma}^\D f \leq M_{r,\sigma}^\D.$$
We will make use of the following lemma from Hyt\"onen and P\'erez \cite[Lemma 2.1]{HyP}.

\begin{lemma} \label{dyadgeom}
Let $\sigma$ be a weight and $0<p<\infty$.  Then $M^\D_{0^+,\s}$ is bounded on $L^p(\s)$ and 
$$\| M^{\D}_{0^+,\s} \|_{L^p(\s)\ra L^p(\s)} \le e^{\frac{1}{p}}.$$
\end{lemma}


\begin{proof}[Proof of Theorem \ref{dyadic results}]

We will prove that (iv) implies (iii) and (i) implies (iv), the other implications follow from standard arguments (see \cite{CU}).  First we will prove (iv) implies (iii).  Again we will suppose that $f$ is supported on a fixed cube and $f>0$ on that cube $Q_0$.  For each integer $k$, let $A_k = \{ x \in \R^n : 2^k < M^\D_{-1} f(x) \le 2^{k+1} \}$. Let $\Sc_k$ be the set of cubes $Q\in \D$ that are maximal with respect to inclusion and satisfy
	$$2^k < \left(\dashint_{Q} f^{-1} \right)^{-1}.$$
 Then each $Q$ in $\Sc_k$ is contained in $Q_0$ and we also have $A_k\subset \bigcup_{Q\in \Sc_k}Q$.  Define $\Sc=\bigcup_k \Sc_k$.  Moreover, given $Q\in \Sc_k$ define $E(Q)=Q\cap A_k$.  Since the cubes $Q\in \Sc_k$ are disjoint for each $k$ and the families $A_k$ are disjoint in $k$, the family $\{E(Q)\}_{Q\in \Sc}$ will be pairwise disjoint and satisfy $A_k=\bigcup_{Q\in \Sc^k}E(Q).$
 
 Then we have
	
\begin{align*}
\int_{\R^n} \left( M^\D_{-1}(f\s^{-1}) \right)^p u\, dx &= \sum_k \int_{A_k} \left( M^\D_{-1}f \right)^p u\, dx\\
&\le \sum_k u(A_k) 2^{p(k+1)}\\
&\lesssim \sum_{Q\in \Sc}\left(\dashint_Q f^{-1}\s\right)^{-p} u(E(Q))\\
&=\sum_{Q\in \Sc}\left(\frac{1}{\s(Q)}\int_Q f^{-1}\s\right)^{-p} \left(\frac{|Q|}{\s(Q)}\right)^{p}u(E(Q))\\
&=\int_0^\infty \mu(\{ Q\in \Sc: F(Q)>\lambda\})\,d\lambda
\end{align*}
where for $Q\in \Sc$, 
$$\mu(Q)=\left(\frac{|Q|}{\s(Q)}\right)^{p}u(E(Q)), \quad \text{and} \quad F(Q)=\left(\frac{1}{\s(Q)}\int_Q f^{-1}\s\right)^{-p}.$$
We have
$$\bigcup_{\substack{Q\in \Sc \\ F(Q)>\lambda}} Q\subset \{x: M_{-1,\sigma}^\D f(x)>\lambda\}.$$
Moreover, if we let $\{Q_i\}$ be the set of maximally dyadic cubes in the set $\{Q\in \Sc:F(Q)>\lambda\}$ then this will for a pairwise disjoint set. Using the testing condition (iv) we see that 
\begin{align*}
\int_0^\infty \mu(\{ Q\in \Sc: F(Q)>\lambda\})\,d\lambda &=\int_0^\infty \sum_{\substack{Q\in \Sc \\ F(Q)>\lambda}} \mu(Q) \,d\lambda\\
 &=\int_0^\infty\sum_{\substack{Q\in \Sc \\ F(Q)>\lambda}}\left(\frac{|Q|}{\s(Q)}\right)^{p}u(E(Q)) d\lambda\\
 &=\int_0^\infty \sum_{i}\sum_{\substack{Q\in \Sc \\ Q\subseteq Q_i}}\left(\frac{|Q|}{\s(Q)}\right)^{p}u(E(Q))d\lambda\\
&\leq \int_0^\infty \sum_{i}\sum_{\substack{Q\in \Sc \\ Q\subseteq Q_i}}\left(\int_{E(Q)}M_{-1}^\D(\1_{Q_i}\sigma^{-1})^pu\right)d\lambda\\
&\lesssim\int_0^\infty \sum_{i}\sigma(Q_i)d\lambda\\
&=\int_0^\infty \sigma(\{x: M_{-1,\sigma}^\D f(x)>\lambda\})\,d\lambda\\
&=\int_{\R^n}(M_{-1,\sigma}^\D f)^p\,\sigma\\
&\lesssim \int_{\R^n} f^p\sigma.
\end{align*}
The limiting argument to remove the support condition finishes the proof of (iv) implies (iii).   We have also used the fact that $M_{-1,\sigma}^\D$ is bounded on $L^p(\sigma)$ (see Lemma \ref{dyadgeom}).  We now prove that (i) implies (iv).  Let $Q\in \D$, we may assume $\sigma(Q)>0$, since otherwise by (i) we have $u(Q)=0$.  For each $\lambda>0$ let 
$$E_\lambda=\{x\in Q:M_{-1}^\D(\1_Q\sigma^{-1})(x)>\lambda\}.$$
Furthermore, let $R=\frac{|Q|}{\sigma(Q)}$.  Then
\begin{multline*}
\int_Q M_{-1}^\D(\1_Q\sigma^{-1})^p u=p\int_0^R\lambda^{p-1}u(E_\lambda)\,d\lambda+p\int_R^\infty\lambda^{p-1}u(E_\lambda)\,d\lambda \\
=I+II
\end{multline*}
The first term is easy to estimate:
$$I=p\int_0^R\lambda^{p-1}u(E_\lambda)\,d\lambda\leq u(Q)\left(\frac{|Q|}{\sigma(Q)}\right)^p\leq C\sigma(Q).$$
For the second term let $\{Q^\lambda_i\}$ be the collection of maximal dyadic cubes such that 
$$\left(\dashint_{Q_i} \sigma\right)^{-1}>\lambda$$
so that $E_\lambda=\bigcup_i Q^\lambda_i$. Then 
\begin{align*}
II&=p\int_R^\infty \lambda^{p-1}u(E_\lambda)\,d\lambda\\
&=p\int_R^\infty \lambda^{p-1}\sum_i u(Q^\lambda_i)\,d\lambda \\
&\lesssim p\int_R^\infty \lambda^{p-1}\sum_i |Q^\lambda_i|\left(\,\dashint_{Q_i^\lambda} \sigma\right)^{p+1}\,d\lambda \\
&\leq p\int_R^\infty \lambda^{-2}\sum_i |Q^\lambda_i|\,d\lambda \\
&\lesssim p|Q| R^{-1}=p\,\sigma(Q).
\end{align*}
This finishes the proof.
\end{proof}

Finally we end our discussion with the statement of specific results for the dyadic geometric maximal operator.  These results were alluded to in \cite{CU}. Using Lemma \ref{approx} for nice functions we have
$$M_{0^-}^{\D} f (x) = M_{0}^\D f(x)$$
where $ M^\D_{0^-} f(x) =\lim_{r\ra 0^+}M_{-r}^\D f(x)$.  We can extend our results to the geometric maximal operator for both the weak and the strong inequalities.  We do not include the proofs as they are similar to that found in \cite{CN2}.

\begin{theorem}\label{geomdyadweak}
Suppose $(u,v)$ is a pair of weights defined on $\R^n$.  Then the following are equivalent:
\begin{enumerate}[(i)]
\item The weak $(p,p)$ inequalities
$$u(\{x\in \R^n:M^\D_0f(x)>\lambda\})\leq \frac{C}{\lambda^p}\int_{\R^n} |f|^pv\,$$
hold for all $0<p<\infty$ and $f\in L^p(v)$,
\item The weak $(1,1)$ inequality  
$$u(\{x\in \R^n:M^\D_0(v^{-1}f)(x)>\lambda\})\leq \frac{C}{\lambda}\int_{\R^n} |f|\,$$
holds for all $f\in L^1(\R^n)$,
\item the pair $(u,v)$ satisfies the two weight $A_\infty$ condition
$$\sup_{Q\in \D} \left(\dashint_Q u\,\right)\exp\left(-\dashint_Q \log v\,\right)<\infty.$$ 
\end{enumerate}
\end{theorem}

\begin{theorem}\label{geomdyadstrong}
Suppose $(u,v)$ is a pair of weights defined on $\R^n$.  Then the following are equivalent:
\begin{enumerate}[(i)]
\item the inequalities
$$\int_{\R^n} (M^\D_0 f)^pu\leq C\int_{\R^n} |f|^pv\,$$
hold for all $0<p<\infty$ and $f\in L^p(v)$,
\item the  inequality  
$$\int_{\R^n}M^\D_0(v^{-1}f) u\leq {C}\int_{\R^n} |f|\,$$
holds for all $f\in L^1(\R^n)$,
\item the testing condition
$$\int_Q M_0^\D (v^{-1}\1_Q)u\,\leq C|Q|$$
holds for all cubes $Q \in \mathsf D$.  
\end{enumerate}

\end{theorem}

\bibliographystyle{plain}

\begin{thebibliography}{99}

\bibitem{CU} D. Cruz-Uribe, SFO, {\it The mimimal operator and the geometric maximal operator on $\R^n$}, Studia Math. {\bf 144} (2001) 1-37.

\bibitem{CN1} D. Cruz-Uribe, SFO, and C.J. Neugebauer, {\it The structure of the reverse H\"older classes}, Trans. Amer. Math. Soc. {\bf 347} (1995) 2941-2960.

\bibitem{CN2} D. Cruz-Uribe, SFO, and C.J. Neugebauer, {\it Weighted norm inequalities for the geometric maximal operator}, Publ. Mat. {\bf 42} (1998) 239-263.

\bibitem{CNO} D. Cruz-Uribe, SFO, C.J. Neugebauer, V. Olesen, {\it Norm inequalities for the minimal and maximal operator, and differentiation of the integral}, Publ. Mat. {\bf 41} (1997) 577--604.

\bibitem{DMO} J. Duoandikoetxea, F.J. Mart\'in-Reyes, S. Ombrosi ,{\it On the $A_\infty$ condition for general bases}, preprint.

\bibitem{GLPT} L. Grafakos, L. Liu, C. P\'erez, and R.H. Torres, {\it The multilinear strong maximal function}, J. Geometric Anal. {\bf 21} (2011) 118--149.

\bibitem{HP} P. Hagelstein and I. Parissis, {\it Weighted Solyanik estimates for the Hardy-Littlewood maximal operator and embedding of $A_\infty$ into $A_p$}, to appear J. Geometric Anal. \textsf{http://arxiv.org/abs/1405.6631}.

\bibitem{HP2} P. Hagelstein and I. Parissis, {\it Weighted Solyanik estimates for the strong maximal function}, preprint \textsf{http://arxiv.org/abs/1410.3402}.

\bibitem{HLP} P. Hagelstein, T. Luque, and I. Parissis, {\it Tauberian conditions, Muckenhoupt weights, and differentiations properties of weighted bases}, Trans. Amer. Math. Soc. {\bf 367} (2015), 7999-8032.

\bibitem{HyP} T. Hyt\"onen and C. P\'erez, {\it Sharp weighted bounds involving $A_\infty$}, Anal. PDE {\bf 6} (2012), 777--818.

\bibitem{M} K. Moen, {\it Sharp weighted bounds without testing or extrapolation}, Arch. Math. (Basel) {\bf 99} (2012), 457--466.

\bibitem{J} B. Jawarth, {\it Weighted inequalities for maximal operators: linearization, localization, and factorization}, Amer. J. Math. {\bf 108} (1986), 361--414.

\bibitem{LL} L. Liguang and T. Luque, {\it A $B_p$ condition for the strong maximal function}, Trans. Amer. Math. Soc. {\bf 366} (2014), 5707--5726.

\bibitem{P1} C. P\'erez, {\it A remark on weighted inequalities for general maximal operators}, Proc. Amer. Math. Soc. {\bf 119} (1993), 1121--1126.


\bibitem{P2} C. P\'erez, {\it Weighted norm inequalities for general maximal operators}, Publ.Mat. {\bf 34} (1990), 1121--1126.

\bibitem{Shi} X. Shi, {\it Two inequalities related to geometric mean operators}, J. Zhejiang Teacher's {\bf 1} (1980) 21-25.

\bibitem{YM} X. Yin and B. Muckenhoupt, {\it Weighted inequalities for the maximal geometric mean operator}, Proc. Amer. Math. Soc. {\bf 124} (1996), 75-81.

\end{thebibliography}

\end{document}